\documentclass[aop]{imsart}

\usepackage{amsthm,amsmath,amssymb,natbib}

\startlocaldefs

\newtheorem{theo}{Theorem}[section]
\newtheorem{prop}[theo]{Proposition}
\newtheorem{lem}[theo]{Lemma}
\newtheorem{cor}[theo]{Corollary}

\theoremstyle{remark}
\newtheorem*{rem}{Remark}

\theoremstyle{definition}

\newtheorem*{hyp}{Hypothesis($\lambda_0$)}

\newcommand{\dcb}{\begin{array}{lll}}
\newcommand{\dce}{\end{array}}

\def\R{\mathbb{R}}
\def\1{\mathbf{1}}
\def\E{\mathbb{E}}

\def\N{\mathbb{N}}

\def\G{\mathtt{G}}

\def\m{m}
\def\const{\mathrm{const}}

\def\F{\cal{F}}

\endlocaldefs

\begin{document}

\begin{frontmatter}

\title{Poincar\'e inequality and exponential integrability of hitting times for linear diffusions}
\runtitle{Poincar\'e inequality and hitting times}

\begin{aug}
\author{\fnms{Dasha} \snm{Loukianova}\corref{}\ead[label=e2]{dasha.loukianova@univ-evry.fr}\thanksref{m2}},
\author{\fnms{Oleg} \snm{Loukianov}\ead[label=e1]{oleg@iut-fbleau.fr}\thanksref{m1,m2}}
\and
\author{\fnms{Shiqi} \snm{Song}\ead[label=e3]{shiqi.song@univ-evry.fr}\thanksref{m2}}
\affiliation{D\'epartement de Math\'ematiques, Universit\'e d'Evry\thanksmark{m2}}
\affiliation{IUT de Fontainebleau, Universit\'e Paris-Est\thanksmark{m1}}

\address{Bd Fran\c{c}ois Mitterrand\\ 91025 Evry, France\\ \printead{e2}\\ \phantom{E-mail:\ }\printead*{e3}}
\address{Route Hurtault\\ 77300 Fontainebleau, France\\ \printead{e1}}

\runauthor{Loukianov, Loukianova and Song}
\end{aug}

\begin {abstract}
Let $X$ be a regular linear continuous positively recurrent Markov process with state space $\R$, scale function $S$ and speed measure $m$. For $a\in \R$ denote
\begin{align*}
B^+_a&=\sup_{x\geq a} \m(]x,+\infty[)(S(x)-S(a))\\
B^-_a&=\sup_{x\leq a} \m(]-\infty;x[)(S(a)-S(x))
\end{align*}
We study some characteristic relations between $B^+_a$, $B^-_a$, the exponential moments of the hitting times $T_a$ of $X$, the Hardy and Poincar\'e inequalities for the Dirichlet form associated with $X$. As a corollary, we establish the equivalence between the existence of exponential moments of the hitting times and the spectral gap of the generator of~$X$.
\end{abstract}



\end{frontmatter}

\section*{Introduction}

In this paper $\R$ is considered as a metric space equipped with its usual Borel field. All functions or measures mentioned below are supposed to be Borel measurable.

Let $(X_t\ ;\ t\geq 0)$ be a regular linear continuous Markov process with the state space $\R$. We assume throughout the paper that $X$ is positively recurrent and conservative (the killing time is identically $+\infty$). Denote by $S(x)$ a scale function of $X$ and $m(dx)$ the speed measure associated with $S$ (cf. \cite[ch.VII]{RY}). Recall that $S$ is a continuous strictly increasing function and $m(dx)$ is a symmetric measure for $X$, charging every no empty open set. Moreover, the positive recurrence of $X$ implies $\lim_{x\to \pm \infty}S(x)=\pm\infty$ and $m(\R)<\infty$.

Let $a\in\R$ and $T_a=\inf\{t\geq 0\ :\ X_t=a\}$ be the hitting time by $X$ at~$a$. The first question we are interested in is the existence of exponential moments $\mathbb{E}_x[e^{\lambda T_a}]$, $x\in\R$, $\lambda>0$. 

Hitting times of linear Markov process intervene in many circumstances: mathematical finance, neural modeling, sequential analysis in statistics, etc. The finiteness of their exponential moments  permits to obtain  moderate and large deviations for additive functionals of $X$, important in all considerations using averaging principle.  In some particular cases such moments  have been well studied. We mention, for example, Ditlevsen \cite {Dit} for Ornstein-Uhlenbeck process, Giorno, Nobile, Riccardi, Sacredote \cite {GNRS}  for Bessel and Ornstein-Uhlenbeck processes, Deaconu and Wantz \cite {DeaWa} for diffusion with strong drift, and the book of Borodin and Salminen \cite {BoSa} for an overview of known formulas. But we were not able to find in the literature a simple general criterion of the existence of exponential moments in terms of the scale function $S(x)$ and the speed measure $m(dx)$. In the present paper this question gets a very satisfactory response in the quantities:
\begin{align*}
B^+_a&=\sup_{x\geq a} \m(]x,+\infty[)(S(x)-S(a))\\
B^-_a&=\sup_{x\leq a} \m(]-\infty;x[)(S(a)-S(x))
\end{align*}
Namely, let $\lambda_a^+$  be the supremum of $\lambda>0$ such that $\E_x e^{\lambda T_a}<\infty$ for some $x>a$ (hence for all $x>a$, see the ``all-or-none'' property~\ref {allornothing}). Respectively, let $\lambda_a^-$ be the supremum of $\lambda$ such that $\E_x e^{\lambda T_a}<\infty$ for some (hence all) $x<a$. Our first result (see section~\ref{sec:expmom}, theorem~\ref {Expmoments}) asserts that
\begin{eqnarray*}
\frac{1}{4B_a^+}\leq\lambda^+_a\leq \frac{1}{B_a^+}\\
\frac{1}{4B_a^-}\leq\lambda^-_a\leq \frac{1}{B_a^-}
\end{eqnarray*}
where $B_a^+$ and $B_a^-$ can eventually be infinite.

Actually, quantities similar to $B^+_a$, $B^-_a$ have already appeared in a theorem due to M. Artola, G. Talenti and G. Tomaselli \cite {Mu} to characterize a couple of probabilities on $\R$ satisfying some Hardy-type inequality. This theorem was generalized by Bobkov and G\"otze \cite {BoGo} and  Malrieu and Roberto \cite {MalRob} and used to characterize probability measures $\mu$ on $\R$, absolutely continuous with respect to Lebesgue measure and satisfying Poincar\'e and Log-Sobolev inequalities associated with $\int_{\R} (f')^2(x)d\mu(x)$. It turns out that $B^+_a$ and $B^-_a$ are an important characteristics of the process $X$ also in this context. In the second section of our paper 
we bridge  $B_a^+$ and $B_a^-$ to the Hardy and Poincar\'e inequalities associated with  ${\cal E(F)}=\int\left (\frac {dF}{dS}\right )^2 dS$ which, as shown in section~\ref{sec:gap}, is the Dirichlet form associated with $X$. 
 We prove that the best possible constants $A_a^+$, $A_a^-$ in the Hardy inequalities
\begin {eqnarray*}
\int_{a}^\infty (F(x)-F(a))^2d\m(x)\leq A_a^+\int_a^{\infty} \left(\frac{dF}{dS}\right)^2(t)dS(t)\\
\int_{-\infty}^a (F(x)-F(a))^2d\m(x)\leq A_a^-\int_{-\infty}^{a} \left(\frac{dF}{dS}\right)^2(t)dS(t)
\end{eqnarray*}
on an appropriate functional space $\mathcal{F}$, satisfy (see theorem \ref{Hardy})
\begin{eqnarray*}
B_a^+\leq A_a^+\leq 4B_a^+\\
B_a^-\leq A_a^-\leq 4B_a^-
\end{eqnarray*}
Furthermore, if $c_P$ is the best possible constant $c$ in the Poincar\'e inequality 
\begin {equation*}
\int_{-\infty}^{+\infty}\left (F(x)-\frac{m(F)}{m(\R)}\right )^2 dm(x)\leq c\int_{-\infty}^{+\infty}\left (\frac {dF}{dS}\right )^2(x)dS(x),
\end{equation*} 
then $c_P$ satisfies (see theorem~\ref{Poincare})
\begin{equation*}
\sup_{a}(A_a^+\wedge A_a^-)\leq c_P\leq \inf_a(A_a^+\vee A_a^-).
\end{equation*}

In the third section we prove (see theorem~\ref{D-form of X}) that the right-hand side ${\cal E(F)}=\int_{-\infty}^{+\infty}\left (\frac {dF}{dS}\right )^2(x) dS(x)$ of the Poincar\'e inequality, is the Dirichlet form associated with $X$. The Poincar\'e inequality yields then in a usual way a bound on the spectral gap $\gamma=1/c_P$ of the generator of $X$ on $\R$. In their turn, the Hardy inequalities are shown to be related to the spectral gaps $\gamma_a^+$, $\gamma_a^-$ of the generator of $X$ killed at $T_a$ by the equalities
\begin{equation*}
\frac{1}{A^+_a} = \gamma^+_a\quad\mbox{and}\quad
\frac{1}{A^-_a} = \gamma^-_a.
\end{equation*}

At this stage, let us cite a theorem of Carmona and Klein \cite {CaKl} asserting that if the generator of a Markov process admits a spectral gap, then its hitting times have exponential moments. Our results show that, in our setting, these properties are actually both equivalent to the finiteness of $B_a^{+}$ and $B_a^{-}$ for some (and hence for all) $a\in\R$.
We establish thereby the equivalence between the existence of a spectral gap of the generator and exponential moments of hitting times for linear continuous positively recurrent Markov processes.

Finally, in the last section of this article we precise this equivalence binding in a very direct way the exponential moments of hitting times to spectral gaps associated with $X$. Namely, we  show (theorem \ref{gapandmoments}) that for any $a\in\R$, 
\begin{equation*}
\gamma^+_a = \lambda^+_a\quad\mbox{and}\quad\gamma^-_a = \lambda^-_a.
\end{equation*}

A similar identity for exit times from a bounded domain $D$ is actually well known since the works of Khasminskii~\cite {Khas} and Friedman~\cite {Fri}. Namely, if $\tau$ is the exit time from $D$ and $X^D$ is a process killed at $\tau$, then the equality holds between the width of the spectral gap of the generator of $X^D$  and the supremum of  $\lambda>0$ such that $\E_xe^{\lambda \tau}<\infty$ for all $x\in D$.

Section~\ref{sec:friedman} is thus devoted to the proof of this equality for half-spaces $]a;+\infty[$ and $]-\infty;a[$. It should be pointed out that we can not directly apply PDE methods of \cite {Khas} that require $\E_x\tau$ to be bounded. Instead, we use the spectral calculus, which is available thanks to the symmetry of the generator, automatically fulfilled in dimension one.

Notice that we can not establish such kind of equality  directly on $\R$, because the process $X$ is conservative and the exit time from $\R$ is identically infinite. But using the bounds on the optimal Poincar\'e constant $c_P$ above, we relate the global spectral gap $\gamma=1/c_P$ to $\lambda_a^\pm$ (see theorem \ref{gapandmoments}) by the inequalities 
  
\begin{equation*}
\sup_a\left({\lambda^+_a}\wedge{\lambda^-_a}\right) \leq \gamma\leq \inf_a\left({\lambda^+_a}\vee{\lambda^-_a}\right).
\end{equation*}
 

\section {Exponential integrability of hitting times}\label{sec:expmom}

In this section we study the exponential moments of hitting times $T_a$. For $a\in\R$, denote $$\lambda_a^+=\sup\{\lambda\geq 0\ :\ \forall x>a,\quad \E_x e^{\lambda T_a}<\infty\};$$
and
$$\lambda_a^-=\sup\{\lambda\geq 0\ :\ \forall x<a,\quad \E_x e^{\lambda T_a}<\infty\}.$$
As we will see (proposition \ref {allornothing}), an important ``all-or-none'' property holds for any $\lambda>0$:
$$\exists x>a,\quad\E_x e^{\lambda T_a}<\infty \iff \forall x>a,\quad \E_x e^{\lambda T_a}<\infty,$$
the same being true for $x<a$.

Recall the definitions
\begin{align*}
B^+_a&=\sup_{x\geq a} \m(]x,+\infty[)(S(x)-S(a))\\
B^-_a&=\sup_{x\leq a} \m(]-\infty;x[)(S(a)-S(x))
\end{align*}
The main result of this section is
\begin{theo}\label {Expmoments}
For all $a\in\R$,
\[
\frac{1}{4B_a^+}\leq\lambda^+_a\leq \frac{1}{B_a^+}\quad\mbox{and}\quad \frac{1}{4B_a^-}\leq\lambda^-_a\leq \frac{1}{B_a^-}
\]
with the convention $1/\infty=0$.
\end{theo}

In the sequel we often prove only assertions concerning $B_a^+$ and $\lambda_a^+$, since the proofs of their counterparts are completely similar.

\subsection{Kac formula}
The Kac formula, first derived in Kac~\cite{Kac,Kac1} for linear Brownian motion, then generalized in Darling and Kac~\cite {DK} and in Fitsimons and Pitman~\cite{FitsKac}, permits to calculate the moments of $A_v=\int_0^Tv(X_t)dt$ for a function $v$ of a Markov process $(X)$ and a suitable random time $T$. In our proof we need a particular case of this formula, where $v=1$ and $T$ is an exit time from an interval or a hitting time.

For  $a<x<b$ consider 
\begin{equation*}
T_{a,b}=\inf\{t\geq 0; \ X_t\notin ]a,b[\}.
\end{equation*}
The Green potential kernel on $[a,b]$ is given by (see e.g. \cite[ch. VII]{RY})
\begin{equation*}
G(a,b,x,y)=\frac{1}{S(b)-S(a)}\left\{
\begin{array}{ccl}
(S(b)-S(x))(S(y)-S(a))&\mbox{if}&a\leq y\leq x\leq b\\
(S(b)-S(y))(S(x)-S(a))&\mbox{if}&a\leq x\leq y\leq b
\end{array}\right.
\end {equation*}
This kernel defines  the Green operator
\[Gf(x)=\int_a^b G(a,b,x,y)f(y)\,dm(y).\]
Notice that since $G(a,b,x,a)=G(a,b,x,b)=0$, the integration interval may or may not include $a$ and $b$.

With the help of this operator we can calculate the moments of $T_{a,b}$ using Kac formula
\begin{equation}\label {Tnab}
\E_x T_{a,b}^n=n\int_a^b G(a,b,x,y)\E_y T_{a,b}^{n-1} dm(y)=n! G^n 1(x)
\end {equation}
To obtain an analogous formula  for the moments of hitting times $T_a$, $a\in\R$,  recall that $\lim_{t\to\pm\infty} S(t)=\pm\infty$, and consider the limits of $G(a,b,x,y)$ when $a\to -\infty$  (resp. $b\to\infty$ ): 

\begin{equation*}
G(-\infty,b,x,\xi)=\left\{\begin{array}{cr}
(S(b)-S(\xi))&  x\leq \xi\leq b\\
(S(b)-S(x))&-\infty< \xi \leq x
\end{array}\right.
\end{equation*}
\begin{equation*}
G(a,+\infty,x,\xi)=\left\{\begin{array}{ll}
(S(\xi)-S(a))&  a\leq \xi\leq x\\
(S(x)-S(a))& x \leq \xi< \infty
\end{array}\right.
\end{equation*}
Taking monotone limits in \eqref {Tnab}, we get a formula for the $n$-th moment of hitting times (see also \cite {LLL}):
\begin{equation}\label{eq:Kacpoly}
\begin{split}
\E_xT_{b}^n=n\int_{-\infty}^{b}G(-\infty,b,x,\xi) \E_{\xi}T_{b}^{n-1} dm(\xi)\quad\mbox{if}\quad x<b\\
\E_xT_{a}^n=n\int_{a}^{+\infty}G(a,+\infty,x,\xi) \E_{\xi}T_{a}^{n-1} dm(\xi)\quad\mbox{if}\quad x>a
\end{split}
\end{equation}
The summation over $n$ yields a formula for exponential moments:
\begin{equation}\label{eq:Kacexp} 
\begin{split}
\E_x\exp (\lambda T_b)=1+\lambda\int_{-\infty}^{b}G(-\infty,b,x,\xi)\E_{\xi}\exp(\lambda T_{b}) dm(\xi),\quad x<b\\
\E_x\exp(\lambda T_a)=1+\lambda\int_{a}^{+\infty}G(a,+\infty,x,\xi)\E_{\xi} \exp(\lambda T_{a})  dm(\xi),\quad x>a
\end{split}
\end{equation}
\begin {rem}
The expressions \eqref{eq:Kacpoly}--\eqref{eq:Kacexp} are always defined, since all functions therein are positive. 
\end{rem}
The following proposition will be referred to as ``all-or-none'' property in the sequel:
\begin{prop}[all-or-none]\label {allornothing}
Let $a\in\R$ and $\lambda>0$. The following properties are equivalent:
\begin{itemize}
	\item for some $x>a$, $\E_{x}\exp(\lambda T_a)<\infty$
	\item for all $x>a$, $\E_x\exp(\lambda T_a)<\infty$
	\item $\int_{a}^{+\infty}\E_{\xi} \exp(\lambda T_{a}) dm(\xi)<\infty$
\end{itemize}

The same holds for $x<a$.
\end{prop}
\begin{proof}
Observe that $G(a,+\infty,x,\xi)\equiv\const>0$ for $\xi>x$, and that $\E_{\bullet}\exp(\lambda T_a)$ is increasing on $]a,\infty[$. Using the exponential Kac formula we then see that for $x>a$,  $\E_x\exp(\lambda T_a)<\infty$ if and only if $\E_\xi\exp(\lambda T_a)$ is $m$-integrable on $]a,\infty[$. In this case, since $m$ charges every interval of $\R$,  $\E_\xi\exp(\lambda T_a)<\infty$ for all $\xi>a$ by monotonicity of $\E_\xi\exp(\lambda T_a)$.
\end{proof}

\subsection {Exit time from an interval}

It is known for a while (Khasminskii condition, see \cite{FitsKac}) that the exit time $T_{a,b}$ from a bounded interval $[a,b]$ admits an exponential moment for some $\lambda>0$. Let 
$$\lambda_{a,b}=\sup\{\lambda\ :\ \forall x\in]a,b[,\ \E_{x}\exp(\lambda T_{a,b})<\infty\}$$
In this subsection we establish some upper and lower bounds for $\lambda_{a,b}$. The upper bound will be particularly important in the proof of the theorem~\ref{Expmoments}. 

\begin{lem}\label {sortint1}$\lambda_{a,b}\geq 1/C$, where
\begin{equation*}
C=\frac{1}{S(b)-S(a)}\int_a^b(S(b)-S(y))(S(y)-S(a)))\,dm(y)
\end{equation*}
\end{lem}
\begin{proof}
Observe that
\[G(a,b,x,y)\leq \frac{(S(b)-S(y))(S(y)-S(a))}{S(b)-S(a)}.\]
If $f$ is positive and bounded, then
\[Gf(x)\leq \frac{1}{S(b)-S(a)}\int_a^b (S(b)-S(y))(S(y)-S(a))f(y)\,dm(y) \leq C\|f\|_\infty.\]
It follows that $\E_x T_{a,b}^n/n!=G^n 1(x)\leq C^n$, whence $\E_x e^{\lambda T_{a,b}}<\infty$ for $\lambda<1/C$.
\end{proof}

To find an upper bound for $\lambda_{a,b}$, fix some interval $[a',b']\subset ]a,b[.$ Define $\kappa_1>0$ and $\kappa_2>0$ by
\begin{equation}
S(a')-S(a)=\kappa_1(S(b')-S(a')),\quad S(b)-S(b')=\kappa_2(S(b')-S(a'))
\label{eq:kappa}
\end{equation}
and denote
\begin{equation*}
{c}=\frac{\kappa_1\kappa_2}{1+\kappa_1+\kappa_2}(S(b')-S(a'))m([a',b']).
\end{equation*}
\
\begin{lem}\label{sortint}
If $\lambda\geq 1/c$ then for all $x\in[a',b']$,  $\E_xe^{\lambda T_{a,b}}=\infty$. In particular, $\lambda_{a,b} \leq 1/c$ for any choice of $[a',b']$.  
\end{lem}
\begin{proof}
Observe that for all $x,y$ in $[a',b']$,
\[G(a,b,x,y)\geq \frac{(S(b)-S(b'))(S(a')-S(a))}{S(b)-S(a)}=\frac{\kappa_1\kappa_2}{1+\kappa_1+\kappa_2}(S(b')-S(a'))\]
 It follows that for all $x\in[a',b']$
\[\E_x T_{a,b}\geq\int_{a'}^{b'}G(a,b,x,y)dm(y)\geq\frac{\kappa_1\kappa_2}{1+\kappa_1+\kappa_2}(S(b')-S(a'))m([a',b'])=c\]
By induction $\E_x T_{a,b}^n\geq n!c^n$, as seen from
\begin{multline*}
\E_x T_{a,b}^n\geq n\int_{a'}^{b'}G(a,b,x,y)\E_y T_{a,b}^{n-1}\,dm(y)\geq\\ n(n-1)!c^{n-1}\int_{a'}^{b'}G(a,b,x,y)\,dm(y)=n!c^n
\end{multline*}
Hence $\E_x e^{\lambda T_{a,b}}=\infty$ pour $\lambda\geq 1/c$ and $x\in[a',b']$.
\end{proof}

In the introduction we have mentioned a theorem of Carmona-Klein~\cite {CaKl}
$$\mbox{``spectral gap''} \Rightarrow \E_x e^{\lambda T_U}<\infty,$$
where $U$ is a set of positive invariant measure. The formulation of this theorem is somewhat confusing, since it does not precise that $\lambda$ depends on~$U$. In fact, the following corollary shows that the property $\E_x e^{\lambda T_a}<\infty$ can not hold simultaneously for all $(x,a)$ with a common $\lambda>0$. 
\begin{cor}\label{cor:nolambda}
$\forall \lambda >0,\ \forall x\in \R$, there exist $a<x$ and $b>x$ such that $\E_xe^{\lambda T_a}=\E_xe^{\lambda T_b}=\infty$.
\end{cor}
\begin{proof}
Fix $\lambda>0$ and $x\in\R$. Put, for example, $\kappa_1=\kappa_2=1$ and chose $[a',b']$ and $[a,b]$ in such a way that $x\in[a',b']\subset]a,b[$ and the equalities~\eqref{eq:kappa} hold. Then
\[\frac{1}{c}=\frac{3}{(S(b')-S(a'))m([a',b'])}<\lambda\]
as soon as $(S(b')-S(a'))m([a',b'])>3\lambda$, which can always be achieved taking $a'$ or $b'$ large enough. Hence, according to lemma~\ref{sortint},  $\E_x e^{\lambda T_{a,b}}=\infty$ and thereby $\E_x e^{\lambda T_{a}}= \E_x e^{\lambda T_{b}}=\infty$ for such $a$ and $b$.
\end{proof}

\subsection {Hitting time} 
In this subsection we will prove the theorem~\ref{Expmoments}:
\[
\frac{1}{4B_a^+}\leq\lambda_a^+\leq \frac{1}{B_a^+}\quad\mbox{and}\quad \frac{1}{4B_a^-}\leq\lambda_a^-\leq \frac{1}{B_a^-}
\]
The proofs of two parts being completely similar, we only give one for $\lambda_a^+$. It will be split in a number of propositions.

\begin{prop}\label {momentinfini}
$\forall a\in\R,\ \lambda_a^+\leq \frac{1}{B_a^+}$, where $1/\infty=0$.
\end{prop}
\begin{proof}
Fix some $a\in\R$. From the lemma \eqref {sortint} we deduce that for $a<a'< x< b'<b$, for all $x\in[a',b']$,   $\E_x e^{\lambda T_{a}}=\infty$ if $\lambda \geq 1/c$, where
\[\frac{1}{c}=\frac{1+\kappa_1+\kappa_2}{\kappa_1\kappa_2(S(b')-S(a'))m([a',b'])}.\]
Now fix $a'$ and $b'$ and make $k_2\to\infty$ (so $b\to\infty$), then
\[\frac{1}{c}\to \frac{1}{\kappa_1(S(b')-S(a'))m([a',b'])}=\frac{1}{(S(a')-S(a))m([a',b'])}.\]
We conclude that $\E_x e^{\lambda T_{a}}=\infty$ for $x\in [a',b']$ and
\[\lambda>\frac{1}{(S(a')-S(a))m([a',b'])}.\]
It follows by the ``all-or-none'' proposition~\ref{allornothing} that $\E_x e^{\lambda T_{a}}=\infty$ for any $x>a$, all $a'$, $b'$ and $\lambda$ as above. Observing that
\[\sup_{a',b'}(S(a')-S(a))m([a',b'])=\sup_{a'}(S(a')-S(a))m([a',\infty[)= B_a^+\]
by the continuity of $S$, we get $\E_x e^{\lambda T_{a}}=\infty$ for any
\[\lambda>\inf_{a',b'}\frac{1}{(S(a')-S(a))m([a',b'])}=\frac{1}{B_a^+}.\]
The inequality $\lambda_a^+\leq {1}/{B_a^+}$ is thereby proved.
\end{proof}

The lower bound $\lambda_a^+\geq(4B_a^+)^{-1}$ requires more work. To simplify the notations we put $B_a^+=B$.

Define for $x>a$ and $f\geq 0$  two  positive linear operators, $J$ and $K$:
\begin{eqnarray*}
Jf(x)=(S(x)-S(a))\int_x^\infty f(y)dm(y)\\ Kf(x)=\int_a^x(S(y)-S(a))f(y)dm(y)
\end{eqnarray*}
where $\int_x^y$ is understood as $\int_{]x,y]}$. Notice that $G=J+K$.

\begin{prop}\label{algebre}
We have
\[\frac{1}{n!}\E_x T_a^n\leq  \sum_{l=0}^n a_{n,l} B^l K^{n-l} 1(x),\]
where $a_{k,l}\geq 0$ satisfy
\[a_{k,l}=0\ \mbox{if}\quad l<0\ \mbox{or}\ l>k,\quad a_{0,0}=1,\quad a_{k+1,l}=\sum_{i\leq l} a_{k,i}\]
\end{prop}

\begin{proof} Using the polynomial Kac formula we see that
\begin{equation*}
\frac{1}{n!}E_x T_a^n= (J+K)^n 1(x).
\end{equation*}
$J$ and $K$ do not commute, hence to handle the above expression we will firstly prove that 
\begin{equation*}
 JKf(x)\leq B(J+K)f(x).
\end{equation*}
for any positive measurable $f$. Indeed,
\begin{multline*}
JKf(x)=(S(x)-S(a))\int_x^\infty  \,d\m(y)\int_a^y(S(u)-S(a))f(u)\,d\m(u)\\
=(S(x)-S(a))\int_x^\infty  \,d\m(y)\int_a^x(S(u)-S(a))f(u)\,d\m(u)\\
+(S(x)-S(a))\int_x^\infty  \,d\m(y)\int_x^y(S(u)-S(a))f(u)\,d\m(u)\\
\leq BKf(x)+(S(x)-S(a))\int_x^\infty f(u)\,d\m(u) (S(u)-S(a))\int_u^\infty \,d\m(y)\\
\leq BKf(x)+BJf(x)
\end{multline*}
By induction, we can easily see that the following inequality holds for all $n\in\N$: 
\begin{equation*}
JK^n 1\le B(B^n+B^{n-1}K 1+\ldots+K^n 1).
\end{equation*}
Now, to prove the proposition we proceed by induction over $n$. For $n=0$ we have $\E_x 1=1=a_{0,0}$. Suppose that
\[\frac{1}{n!}\E_x T_a^{n}\leq  \sum_{l=0}^n a_{n,l} B^l K^{n-l} 1(x),\]
then
\begin{multline*}
\frac{1}{(n+1)!}\E_x T_a^{n+1}=\frac{1}{n!}(J+K)\E_\bullet T_a^n(x)\leq (J+K) \sum_{l\geq 0} a_{n,l}B^l K^{n-l} 1(x)\\
=J\sum_{l\geq 0} a_{n,l}B^l K^{n-l} 1(x)+\sum_{l\geq 0} a_{n,l}B^l K^{n-l+1} 1(x)\\
\leq B\sum_{l\geq 0} a_{n,l}B^{l}\sum_{i=0}^{n-l} B^{n-l-i} K^{i} 1(x) + \sum_{l\geq 0} a_{n,l}B^{l}K^{n-l+1} 1(x)\\
=\sum_{i\geq 0} B^{n+1-i} K^i 1(x) \sum_{l\leq n-i} a_{n,l}+ \sum_{i\geq 0} a_{n,n+1-i}B^{n+1-i}K^{i} 1(x)\\
=\sum_{i\geq 0} B^{n+1-i} K^i 1(x) \sum_{l\leq n+1-i} a_{n,l}=\sum_{j\geq 0} B^{j} K^{n+1-j} 1(x) \sum_{l\leq j} a_{n,l}
\end{multline*}
\end{proof}

An explicit formula for $a_{n,l}$ can be derived, but such a refinement would not improve the estimations we are aiming to obtain.

\begin{lem}\label {ank}
$\sum_{k=0}^n a_{n,k}=a_{n+1,n}\leq 4^n$
\end{lem}
\begin{proof}
We firstly prove by induction that $a_{n,k}\leq 2^{n+k-1}$ for $n\geq 1$. For $n=1$ we have $a_{1,0}=a_{1,1}=1$, so the inequality is satisfied. Further,
\[a_{n+1,k}=\sum_{i\leq k}a_{n,i}\leq \sum_{i=0}^k 2^{n+i-1}\leq 2^{n+k}\]
We deduce  that $\sum_{k=0}^n a_{n,k}=a_{n+1,n}\leq 4^n$, the inequality being true also for $n=0$.
\end{proof}

\begin{prop}\label{prop:4B}
For  $0\leq\lambda<(4B)^{-1}$ and all $x>a,$ it holds:
\[\int_a^{+\infty}\E_x e^{\lambda T_a}d\m(x)\leq \m(]a,\infty[)\left(1-\frac{\lambda}{4B_a^+}\right)^{-1}.\]
Moreover, in this case
\[\E_x e^{\lambda T_a}\leq \frac{\m(]a,\infty])}{\m(]x,\infty[)}\left(1-\frac{\lambda}{4B}\right)^{-1}.\] 
\end{prop}
\begin{proof}
The proposition being obviously true for $B=\infty$, we can suppose that $B<\infty$.

For positive function $f$ denote $\|f\|=\int_{]a,\infty[}f(x)\,d\m(x)$. Observe that
$\|Kf\|\leq B\|f\|$. Indeed, let $f\geq 0$, then
\begin{multline*}
\|Kf\|=\int_a^\infty \, d\m(x)\int_a^x(S(y)-S(a))f(y)d\m(y)=\\
\int_a^\infty dy f(y)\m(y)(S(y)-S(a)) \int_y^\infty \,d\m(x)\leq B\|f\|
\end{multline*}
Combining this inequality with the proposition~\ref {algebre} and lemma~\ref{ank}, we can write
\begin{multline*} 
\frac{1}{n!}\int_a^{+\infty}\E_x T_a^n d\m(x)\leq  \sum_{l=0}^n a_{n,l} B^l \|K^{n-l} 1\|\leq\\
 \sum_{l=0}^n a_{n,l} B^l B^{n-l} \|1\|\leq  4^n B^n \m(]a,\infty[),
\end{multline*}
which implies the first assertion.

The bound on $\E_xe^{\lambda T_a}$ follows by its monotonicity in $x$ from the following estimation: $\forall x >a$,
$$ \E_x e^{\lambda T_a} \m(]x,+\infty[)\leq \int_x^{+\infty}\E_y T_a^n d\m(y)
\leq \int_a^{+\infty}\E_y T_a^n d\m(y).$$
\end{proof}

Finally, the propositions~\ref{momentinfini} and~\ref{prop:4B} jointly imply the assertion of theorem~\ref{Expmoments}, namely
\[\frac{1}{4B_a^+}\leq\lambda_a^+\leq \frac{1}{B_a^+}\quad\mbox{and}\quad \frac{1}{4B_a^-}\leq\lambda_a^-\leq \frac{1}{B_a^-},\]
the inequalities concerning $\lambda_a^-$ being proved in the same way. 

\begin{rem}
It is easy to see that 
\[\exists a\in\R,\ B_a^+<\infty\iff\forall a\in\R,\ B_a^+<\infty.\]
So the theorem~\ref{Expmoments} yields yet another ``all-or-none'' property:
\[\exists a\in\R,\ \lambda_a^+>0\iff\forall a\in\R,\ \lambda_a^+>0,\]
the same being true for $B_a^-$, $\lambda_a^-$.
The corollary~\ref{cor:nolambda} implies, however, that
\[\lim_{a\to\infty}\lambda_a^-=\lim_{a\to-\infty}\lambda_a^+=0.\]
\end{rem}

\section {Hardy and Poincar\'e inequalities}\label{sec:poincare}
In this section we will see that $B_a^\pm$ also play an important role in another context. Our exposition follows essentially the lines of Malrieu and Roberto~\cite{MalRob}, but in more general setting.

Suppose (as above) that $S(x)$ is a strictly increasing continuous function on $\R$, with $\lim_{x\to\pm\infty}=\pm\infty$. Suppose also that $m$ is a positive Borel measure on $\R$, with $m(\R)<\infty$. In this section we do note assume that $S$ and $m$ are the scale function and the speed measure of some process (though they can be). 

Denote by $dS$ the measure induced by $S(x)$. Let $F(x)$ be a real function on $\R$. We shall write $dF\ll dS$, if there exists a function $f(x)$ in $\mathbb{L}^1(dS)$ such that 
$$\int_a^bf(x) dS(x) = F(b)-F(a),\ \forall a<b$$
The function $f(x)$ will be denoted $\frac{dF}{dS}(x)$. Introduce then the function space
\begin{equation}\label{defF}
{\cal F}=\left\{F\in L^2(\m): \quad dF\ll dS, \ \frac {dF}{dS}\in L^2(dS)\right\}.
\end{equation}

Unlike in~\cite{MalRob}, we do not assume that $dS$ and $m$ are absolutely continuous with respect to the Lebesgue measure.

\subsection{Hardy inequality}
For $a\in\R$ and $0\leq A\leq \infty$, we shall call the following inequalities Hardy inequalities associated with the function space $\mathcal{F}$, over the upper half state space $]a,\infty[$ and the lower half state space $]-\infty,a[$, with the constant $A$:

\begin {equation}\label{eq:Hardy}
\begin{split}
\int_{a}^\infty (F(x)-F(a))^2d\m(x)\leq A\int_a^{\infty} \left(\frac{dF}{dS}\right)^2(t)dS(t),\ \forall F\in\mathcal{F}\\
\int_{-\infty}^a (F(x)-F(a))^2d\m(x)\leq A\int_{-\infty}^{a} \left(\frac{dF}{dS}\right)^2(t)dS(t),\ \forall F\in\mathcal{F}
\end{split}
\end{equation}
The constant $A$ will be called a constant to Hardy inequality. Denote by $A^+_a$ (resp. $A^-_a$) the infinum of the constants to Hardy inequality over the upper (resp. the lower) half state space. Notice that, since $F(x)-F(a)=0$ for $x=a$, we can interpret the left side of the inequality \eqref{eq:Hardy} indifferently as
\[\int_{]a,+\infty[} (F(x)-F(a))^2d\m(x)\quad\mbox{or}\quad\int_{[a,+\infty[} (F(x)-F(a))^2d\m(x).\]
Finally recall that for $a\in \R$ 
\begin {equation*}
B^+_a:=\sup_{x\geq a} \m(]x,+\infty[)(S(x)-S(a))
\end{equation*} 
\begin {equation*}
B^-_a:=\sup_{x\leq a} \m(]-\infty;x[)(S(a)-S(x))
\end{equation*}
It is to notice that the quantities $B^\pm_a$ $A^\pm_a$ can be infinite. Nevertheless the following relations hold between these quantities:
 
\begin {theo}\label{Hardy} For any $a\in\R$ we have the inequalities
\begin{eqnarray*}
B_a^+\leq A_a^+\leq 4B_a^+\\
B_a^-\leq A_a^-\leq 4B_a^-
\end{eqnarray*}
\end{theo}

\begin {proof}
Let $a$ be fixed. We shall only prove the inequalities concerning $A^+_a$. Those concerning $A^-_a$ can be obtained by symmetry. Denote, for simplicity, $B_a^+=B$ and $ A_a^+=A$.

Firstly suppose that $B<\infty$. Using the Cauchy-Schwartz inequality we can write, for $x>a$,
\begin{multline*}
\left (F(x)-F(a)\right)^2 \leq\int_a^x \left (\frac {dF}{dS}\right )^2(t)\sqrt{S(t)-S(a)}dS(t)
\times \int_a^x \frac{dS(t)}{\sqrt{S(t)-S(a)}}\\
=\int_a^x \left (\frac {dF}{dS}\right )^2(t)\sqrt{S(t)-S(a)}dS(t)
\times 2\sqrt{S(x)-S(a)}
\end {multline*}

Applying the Fubini theorem we get
\begin{multline}\label{Fub}
\int_{a}^\infty\left (F(x)-F(a)\right)^2 dm(x)\leq\\
2\int_a^{\infty} \left (\frac {dF}{dS}\right)^2(t)\sqrt{S(t)-S(a)} \left(\int_{]t,+\infty[}{\sqrt{S(x)-S(a)}}dm(x)\right)dS(t)
\end{multline}
Put $M(t)=\m(]-\infty;t])$ and $M=\m(\R)$. The definition of $B$ yields
\begin{equation*}
m(]x,+\infty[)=M-M(x)\leq \frac{B}{S(x)-S(a)}
\end{equation*}
whence
\begin{equation*}
(M-M(x))\sqrt{S(x)-S(a)}\leq \frac{B}{\sqrt{S(x)-S(a)}}
\end{equation*}
Observe that $\lim_{x\to+\infty}(M-M(x))\sqrt{S(x)-S(0)}=0$. The integration by parts formula then yields
\begin{multline*}
\int_{]t;\infty[}{\sqrt{S(x)-S(a)}}dm(x)=-\int_{]t;\infty[} \sqrt{S(x)-S(a)}d(M-M(x))\\
=(M-M(t))\sqrt{S(t)-S(a)}+\int_t^{+\infty}(M-M(x))d\sqrt{S(x)-S(a)}\\
\leq\frac{B}{\sqrt{S(t)-S(a)}}+\frac{B}2\int_t^{+\infty}\frac{dS(x)}{(S(x)-S(a))^{3/2}}=\frac{2B}{\sqrt{S(t)-S(a)}},
\end{multline*}
which, together with~\eqref{Fub} implies
\begin{equation*}
\int_{]0;+\infty[}\left (F(x)-F(0)\right)^2 dm(x)\leq 4B\int_0^{+\infty}\left( \frac {dF}{dS}\right)^2 (t)dS(t).
\end{equation*}
Hence the Hardy inequality~\eqref{eq:Hardy} holds with the constant $4B$, which implies $A\leq 4B$.

Next, suppose $A<\infty$. Take $r>a$ and put $F(x)=S(x\wedge r)-S(x\wedge a)$. $F(x)$ is an element of $\mathcal{F}$, so we can write the Hardy inequality~\eqref{eq:Hardy} for such a $F(x)$:
\begin {multline*}
\left(S(r)-S(a)\right)^2m(]r;+\infty[)\leq \int_{]a,+\infty[}\left(F(x)-F(a)\right)^2 dm(x)\\
\leq A\int_a^{+\infty}\left(\frac {dF}{dS}\right)^2(t) dS(t)=A\int_a^{r}dS(t)=A\left (S(r)-S(a)\right),
\end{multline*}
whence
\[\left(S(r)-S(a)\right)m(]r;+\infty[)\leq A\]
for any $r>a$. This implies $B\leq A$. The inequalities concerning $A^+_a$ follow from these facts.
\end{proof}

\subsection {Poincar\'e inequality}

Let $c\leq \infty$ be a constant. We call the following inequality Poincar\'e inequality associated with the function space $\mathcal{F}$, with the constant $c$:
\begin {equation}\label {eq:Poincare}
\int_{-\infty}^{+\infty}\left (F(x)-\overline{m}(F)\right )^2 dm(x)\leq c\int_{-\infty}^{+\infty}\left (\frac {dF}{dS}\right )^2(x)dS(x), \forall F\in\mathcal{F}
\end{equation}
where $\overline{m}(F)=\frac{1}{m(\R)}\int F(x) dm(x)$. The constant $c$ will be called a constant to Poincar\'e inequality. Denote by $c_P$ the lower bound of the constants to Poincar\'e inequality. It is to notice that $c_P$ itself is a constant to Poincar\'e inequality.

\begin {theo}\label {Poincare}
The following relations hold: 
$$\sup_a(A^+_a\wedge A^-_a) \leq c_P\leq \inf_a(A^+_a\vee A^-_a) .$$
Moreover, $c_P<\infty$, if and only if the constants $A^+_a$, $A^-_a$ (or equivalently $B^+_a$, $B^-_a$) are all finite for some $a\in\R$.
\end {theo}

\begin {proof}
The variational formula for the variance and the Hardy inequality give, for all $F\in\mathcal{F}$ and $a\in \R$,
\begin {multline*}
\int_{-\infty}^{+\infty}\left (F(x)-\overline{m}(F)\right )^2 dm(x)\leq \int_{-\infty}^{+\infty}\left (F(x)-F(a)\right )^2 dm(x)=\\
= \int_{-\infty}^{a}\left (F(x)-F(a)\right )^2 dm(x)+\int_{a}^{+\infty}\left (F(x)-F(a)\right )^2 dm(x)\\
\leq A^-_a \int_{-\infty}^a\left (\frac {dF}{dS}\right)^2(t)dS(t)+A^+_a \int_a^{+\infty}\left (\frac {dF}{dS}\right)^2(t)dS(t)\\
\leq  (A^-_a\vee A^+_a) \int_{-\infty}^\infty \left (\frac {dF}{dS}\right)^2(t)dS(t)
\end {multline*} 
which is just the Poincar\'e inequality with the constant $A^-_a\vee A^+_a$. This being true for all $a$, we get $c_P\leq \inf_a(A^-_a\vee A^+_a)$.

Let us show that $c_P\geq \sup_a(A^+_{a}\wedge A^-_{a})$. Fix some $a\in\R$. For any $A^+<A_a^+$ and $A^-<A_a^-$ there exist some $F_+\in\F$ and $F_-\in\F$  such that 
\begin {equation*}
\int_{a}^{\infty}\left (F_+(x)-F_+(a)\right)^2\,d\m(x)\geq A^+\int_{a}^{\infty}\left (\frac {dF_+}{dS}\right)^2(t)dS(t)>0
\end{equation*}
\begin {equation*}
\int_{a}^{\infty}\left (F_-(x)-F_-(a)\right)^2\,d\m(x)\geq A^-\int_{-\infty}^a\left (\frac {dF_-}{dS}\right)^2(t)dS(t)>0
\end{equation*}
Choose $\alpha\in\R$, $\beta\in\R$ and set 
$$
F=\beta (F_-(x)-F_-(a))\mathbb{I}_{\{x<a\}}+\alpha (F_+(x)-F_+(a))\mathbb{I}_{\{x>a\}}
$$
in such a way that $\alpha^2+\beta^2\neq 0$ and $\overline{m}(F)=0$. Remark that $$
\frac {dF}{dS}|_{]-\infty,a[}=\beta \frac{dF_-}{dS},\quad
\frac {dF}{dS}|_{]a,\infty[}=\alpha \frac{dF_+}{dS}
$$
Then
\begin{multline*}
c_P\int_{-\infty}^{\infty}\left (\frac {dF}{dS}\right)^2(t)dS(t)
\geq \int_{-\infty}^{+\infty} F^2(x)dm(x)\\
=\beta^2\int_{-\infty}^{a} (F_-(x)-F_-(a))^2(x)dm(x)+\alpha^2\int_{a}^{+\infty} (F_-(x)-F_-(a))^2(x)dm(x) \geq\\
\geq \beta^2 A^-\int_{-\infty}^a\left (\frac {dF_-}{dS}\right)^2(t)dS(t)
+\alpha^2 A^+\int_{a}^{\infty}\left (\frac {dF_+}{dS}\right)^2(t)dS(t)=\\
=A^-\int_{-\infty}^a\left (\frac {dF}{dS}\right)^2(t)dS(t)
+A^+\int_{a}^{\infty}\left (\frac {dF}{dS}\right)^2(t)dS(t)\geq\\
\geq A^+\wedge A^-\int_{-\infty}^{\infty}\left (\frac {dF}{dS}\right)^2(t)dS(t)
\end{multline*}
Since $\int_{-\infty}^{\infty}\left (\frac {dF}{dS}\right)^2(t)dS(t)>0$ and $A^+<A_a^+$, $A^-<A_a^-$ are  arbitrary, it follows that $c_P\geq(A^+_{a}\wedge A^-_{a})$ for any $a$, whence $c_P\geq\sup_a(A^+_{a}\wedge A^-_{a})$. The bounds on $c_P$ are proved.

Consider the second assertion of the theorem. If for some $a$, $A^+_a$ and $A^-_a$ are finite, the bounds on $c_P$ imply $c_P<\infty$. Suppose now that $c_P<\infty$. For any $a\in\R$ and $F\in\mathcal{F}$, put $G(x)=F(x)-F(x\wedge a)$. Notice that
\begin{eqnarray*}
\overline{m}(G)^2&\leq& \frac{1}{m(\R)^2}\left(\int_a^\infty(F(x)-F(a))dm(x)\right)^2\\
&\leq&\frac{m(]a,\infty[)}{m(\R)^2}\int_a^\infty(F(x)-F(a))^2dm(x)
\end{eqnarray*}
Using Poincar\'e inequality, we can write
\begin{eqnarray*}
&&\int_a^\infty(F(x)-F(a))^2dm(x)\\
&=&\int_{-\infty}^\infty G^2(x) dm(x)\\
&=&\int_{-\infty}^\infty (G-\overline{m}(G))^2(x) dm(x)+\int_{-\infty}^\infty \overline{m}(G)^2 dm(x)\\
&\leq &c_P \int_{-\infty}^\infty \left(\frac{dG}{dS}\right)^2(x) dS(x)+ m(\R)\overline{m}(G)^2\\
&\leq &c_P \int_{a}^\infty \left(\frac{dF}{dS}\right)^2(x) dS(x)+ \frac{m(]a,\infty[)}{m(\R)}\int_a^\infty(F(x)-F(a))^2dm(x)
\end{eqnarray*}
Shifting the last term to the left, the above computation becomes
$$
\int_a^\infty(F(x)-F(a))^2dm(x)\leq \frac{c_P}{1-\frac{m(]a,\infty[)}{m(\R)}}\int_{a}^\infty \left(\frac{dF}{dS}\right)^2(x) dS(x)
$$which is just the Hardy's inequality over the upper half space $]a,\infty[$ with a finite constant. We conclude that $A^+_a<\infty$. The fact $A^-_a<\infty$ can be seen in the same way. The theorem is proved.
\end {proof}

\section {Spectral gap}\label{sec:gap}

In this section we relate Hardy and Poincar\'e inequalities to spectral gaps of the generators of Dirichlet forms associated with $X$. We suppose anew that $S$ and $m$ are the scale function and the speed measure of $X$.

\begin{theo}\label {D-form of X} The diffusion $X$ is $m$-symmetric. The Dirichlet space associated with $X$ is the function space $\mathcal{F}$ given by~\eqref{defF}, and the Dirichlet form has the expression 
\[
\mathcal{E}(F,F)=\int_{-\infty}^\infty \left(\frac{dF}{dS}\right)^2(x)dS(x), \ F\in\mathcal{F}.
\]
\end{theo}

\begin{proof} The form $(\mathcal{E},\mathcal{F})$ is a Dirichlet form as we can check in the way of, for example, Fukushima et al~\cite[p. 6]{FOT}, Example 1.2.2. It remains to show the association of this Dirichlet form to the diffusion $X$.

Let $\hat{\mathcal{E}}$ be the Dirichlet form associated with $X$ with its generator $A$ and its resolvent $(G_\alpha,\alpha>0)$. Let $\lambda>0$ and 
\[
R_\lambda f(x)=\mathbb{E}_x\int_0^\infty e^{-\lambda t}f(X_t)dt
\]
for a bounded function, then $R_\lambda f$ is a version of $G_\lambda f$. Following Ito-McKean~\cite[p. 98]{IMK}, we introduce the diffusion generator 
$$
\dcb
\mathcal{B}=\{f: \mbox{\small $f$ is bounded with compact support}\}\\
\Theta =\lambda - (R_\lambda)^{-1} \quad\mbox{\small on $\mathcal{B}$}
\dce
$$
It is known (see~\cite[p. 117]{IMK}) that $\Theta$ is the differential operator $\Theta = \frac{d}{dm}\frac{d}{dS}$, and we can relate $\Theta$ to $A$ in the following way (with equalities in the sense of $\mathbb{L}^2(m)$)
$$
AR_\lambda f = \lambda R_\lambda f - f = \Theta R_\lambda f = \frac{d}{dm}\frac{d^{\pm} R_\lambda f}{dS}, \ f\in\mathcal{B}
$$

We will need an integration by parts formula which depends on the boundary conditions 
\begin{equation}
\frac{d^\pm R_\lambda f}{dS}(\infty)= \frac{d^\pm R_\lambda f}{dS}(-\infty)=0, \ f\in\mathcal{B}
\label{eq:dRdS}
\end{equation}

Indeed, since $S(\pm\infty)=\pm\infty$, the boundaries $\pm\infty$ are non-exit boundaries: 
$$
\int_0^\infty m((0,x))dS(x)=\infty, \int_{-\infty}^0 m((x,0))dS(x)=\infty
$$
(see Ito-Mckean~\cite[p. 130]{IMK}). In this case, for $-\infty<a<\infty$ and $g(x)=\mathbb{E}_x[e^{-\lambda T_a}]$, 
$$
\lim_{x\rightarrow-\infty} \frac{d^\pm g}{dS}(x)=0,\ \lim_{x\rightarrow\infty} \frac{d^\pm g}{dS}(x)=0
$$
But if $[a,b]$ is a compact support of a $f\in\mathcal{B}$ and $\tau$ is the hitting time of $[a,b]$, we can write 
$$
R_\lambda f(x)=\mathbb{E}_x[e^{-\lambda \tau}R_\lambda f(X_\tau)]=
\left\{
\dcb
\mathbb{E}_x[e^{-\lambda T_a}]R_\lambda f(a),&&x<a\\
\mathbb{E}_x[e^{-\lambda T_b}]R_\lambda f(b),&&x>b\\
\dce
\right.
$$
which yields the boundary conditions~\eqref{eq:dRdS}.

We can now write, for $g=R_\lambda f, f\in\mathcal{B}$,
\begin{eqnarray*}
\infty &>&\hat{\mathcal{E}}(g,g)\\
&=&-\int_{-\infty}^\infty AR_\lambda f(x) R_\lambda f(x) m(dx)\\
&=&-\int_{-\infty}^\infty  R_\lambda f(x) d\frac{d^+R_\lambda f}{dS}(x) \\
&=&-\left.R_\lambda f(x) \frac{d^+R_\lambda f}{dS}(x)\right|_{-\infty}^\infty +\int_{-\infty}^\infty  \frac{d^+R_\lambda f}{dS}(x)d R_\lambda f(x) \\
&=&\int_{-\infty}^\infty  (\frac{d^+R_\lambda f}{dS}(x) )^2d S(x),
\end{eqnarray*}
which means that $g=R_\lambda f\in\mathcal{F}$ and $\hat{\mathcal{E}}(g,g)=\mathcal{E}(g,g)$

Let $h\in \mathcal{F}$. We take for each $N>1$ a function $\varphi\in C^\infty(\mathbb{R})$ such that 
$$
\dcb
\varphi(x)=-(N+1),&&x<-(N+2)\\
\varphi(x)=x,&&-N<x<N\\
\varphi(x)=(N+1),&&x>(N+2)\\
0\leq \varphi'(x)\leq 1,&&\forall x
\dce
$$
The function $\varphi(h)$ is again an element of $\mathcal{F}$. For any $f\in\mathcal{B}$ we compute
\begin{eqnarray*}
\mathcal{E}_\lambda(\varphi(h),R_\lambda f)&=&\int_{-\infty}^\infty \frac{d\varphi(h)}{dS} \frac{d^+R_\lambda f}{dS} dS + \lambda \int_{-\infty}^\infty \varphi(h) R_\lambda f dm\\
&=&\varphi(h) \frac{d^+R_\lambda f}{dS}|_{-\infty}^\infty-\int_{-\infty}^\infty \varphi(h) d\frac{d^+R_\lambda f}{dS} + \lambda\int_{-\infty}^\infty \varphi(h) R_\lambda f dm\\
&=&-\int_{-\infty}^\infty \varphi(h) AR_\lambda f dm + \lambda\int_{-\infty}^\infty \varphi(h) R_\lambda f dm\\
&=&\int_{-\infty}^\infty \varphi(h)  f dm
\end{eqnarray*}
Let $N\rightarrow\infty$. The conditions on the function $ \varphi$ ensure the convergence of the above quantities. With the resolvent $U_\lambda$ of $\mathcal{E}$, we can then write
$$
\mathcal{E}_\lambda(h,R_\lambda f) = \int_{-\infty}^\infty h  f dm = \mathcal{E}_\lambda(h,U_\lambda f) 
$$
which means $U_\lambda f = R_\lambda f$ for all function $f$ bounded with compact support. As bounded operators on $L^2(m)$, $U_\lambda$ and $R_\lambda$ have to be equal. The theorem is proved. \end{proof}

Recall now the usual properties of spectral gaps. Suppose that $\mathcal{E}$ is a Dirichlet form associated with a non negative self-adjoint operator $-A$ on $\mathbb{L}^2(m)$. Let $(E_\xi,\xi\geq 0)$ be the spectral family associated with $-A$. Denote by $H_\xi$ the image space of $E_\xi$ (which is a projection operator). The elements of $H_0$ are those who satisfy $P_t u=u$ for all $t>0$.

Let $(\cdot,\cdot)$ and $\|\cdot\|$ denote respectively the scalar product and the norm in $\mathbb{L}^2(m)$. We know that $-A$ has a spectral gap at $0$ of width at least $\gamma>0$ if and only if the following inequality
\begin{equation}\label{eq:gap}
\gamma\|f-E_0f\|^2=\gamma\int_{]0,\infty[}d(E_\xi f,f)\leq \int_{]0,\infty[}\xi d(E_\xi f,f)= \mathcal{E}(f,f)
\end{equation}
holds for all $f$ in the domain of $\mathcal{E}$.

It is easy to see now that the Poincar\'e inequality~\eqref{eq:Poincare} with constant $c_P$ can be written as~\eqref{eq:gap}.

\begin{theo} The generator of the Dirichlet form associated with $X$ has a spectral gap $\gamma>0$ if and only if $c_P<\infty$. In this case, $\gamma={1}/{c_P}$.
\end{theo}

\begin{proof} 
By the $L^1$-ergodicity of the process $X$ (see e.g.~\cite{BoSa}), the space $H_0$ can contain only constants. As $X$ is a conservative process, $P_t c=c$ for all $t>0$, whence $H_0=\R$. Notice that $\overline{m}$ is the orthogonal projection operator upon $H_0$, i.e. $E_0=\overline{m}$. The equivalence
$$
\gamma\|F-E_0 F\|^2\leq\mathcal{E}(F,F)\quad\iff\quad \gamma\|F-\overline{m}(f)\|^2
\leq \int_{-\infty}^\infty \left(\frac{dF}{dS}\right)^2 dS
$$
proves the theorem.
%
\end {proof} 


Now we address the Hardy inequalities. Let $a\in \R$. Following Fukushima et al~\cite[p. 142]{FOT}, we introduce the space
$$
\mathcal{F}_{]a,\infty[}=\{f\in\mathcal{F}: f(x)=0, \ x\leq a\}
$$
Then, the restriction of the Dirichlet form $\mathcal{E}$ on $\mathcal{F}_{]a,\infty[}$ is a Dirichlet form   is the one  associated with the semigroup $(P^{]a,\infty[}_t)_{t\geq 0}$ of the linear continuous Markov process $X$ killed when it exits $]a,\infty[$. In the sequel we denote this restriction by $\mathcal{E}_{]a,\infty[}$. The killed process is symmetric with respect to the measure $\mathbb{I}_{]a,\infty[}\cdot m(dx)$.

\begin{prop}\label{halfpoincare} The generator of $\mathcal{E}_{]a,\infty[}$ has a spectral gap $\gamma_a^+>0$ if and only if the Hardy inequality~\eqref{eq:Hardy} holds with $A_a^+<\infty$. In this case, $\gamma_a^+={1}/{A_a^+}$.
\end{prop}

\begin{proof} Recall that if $u\in H_0$, $P^{]a,\infty[}_tu=u$. Take a bounded non negative fonction $v$. We have
$$
(u,v)=(P^{]a,\infty[}_tu,v)=(u,P^{]a,\infty[}_tv).
$$
But
$$
\lim_{t\to\infty}P^{]a,\infty[}_tv(x)\leq \|v\|_\infty \mathbb{P}_x[t<T_a]=0. 
$$
due to the positive recurrence property of $X$. We get $(u,v)=0$ for any such function $v$. This means that $u=0$ and therefore $E_0=0$.

Now, for all $F(x)\in\mathcal{F}_{]a,\infty[}$,
$$
\gamma\|F-E_0 F\|^2\leq\mathcal{E}_{]a,\infty[}(F,F)\quad\iff\quad \gamma\|F\|^2
\leq \int_{a}^\infty \left(\frac{dF}{dS}\right)^2 dS.
$$
Clearly, for any $F(x)\in\mathcal{F}$, $F(x)-F(x\wedge a)\in\mathcal{F}_{]a,\infty[}$, which finishes the proof.
\end{proof}

The same property evidently holds for $A_a^-$.
\section {Khasminskii identity}\label{sec:friedman}

The last section binds in a very direct way the exponential moments to the spectral gaps associated with $X$. Namely, we show that
\[\gamma_a^+=\lambda_a^+\quad\mbox{and}\quad \gamma_a^-=\lambda_a^-.\]

We begin with a general remark. Consider a Hunt process $X$ on a Polish space $\mathtt{E}$ in the sense of Fukushima et al~\cite{FOT}. Let $m$ be a Radon measure on $\mathtt{E}$. Suppose that $m$ is bounded and $X$ is a $m$-symmetric process. Denote by $(P_t)_{t\geq 0}$ the transition semigroup of $X$. Denote by $\mathbb{P}_x,x\in\mathtt{E}$, the law of the process $X$ issued from $x\in\mathtt{E}$. For an open set $\mathtt{G}$ in $\mathtt{E}$, set
$$
\tau_\mathtt{G}=\inf\{t>0: X_t\notin \mathtt{G}\}
$$
the exit time of $X$ from $\mathtt{G}$. Introduce 
$$
P^\mathtt{G}_t[A](x)=\mathbb{P}_x[X_t\in \mathtt{A};t<\tau_\mathtt{G}]
$$
for measurable subset $\mathtt{A}$ of $\mathtt{E}$, and set
$$
Y_t=\left\{
\dcb
X_t,&&0\leq t <\tau_\mathtt{G}\\
\Delta && t\geq \tau_\mathtt{G}
\dce
\right.
$$
Then, according to~\cite{FOT}, $Y$ is a Hunt process on the state space $\mathtt{G}$, symmetric with respect to the measure $\mathbb{I}_\mathtt{G}\cdot m(dx)$ with the transition semigroup $(P_t^\mathtt{G})$. If $A^\mathtt{G}$ denotes the infinitesimal generator of $(P_t^\mathtt{G})$ in $\mathbb{L}^2(\mathbb{I}_\mathtt{G}\cdot m(dx))$, $A^\mathtt{G}$ is a self-adjoint negative operator. Let us denote by $(\cdot,\cdot)$ the scalar product in $\mathbb{L}^2(\mathbb{I}_\mathtt{G}\cdot m(dx))$ and by $(E_\xi,\xi\geq 0)$ the spectral measure of $-A^\mathtt{G}$. 

For any bounded non negative function $f(x)\in\mathbb{L}^2(\mathbb{I}_\mathtt{G}\cdot m(dx))$, for all $\lambda> 0$, $0<N<\infty$, we have the formulas
\begin{eqnarray*}
\int_0^N e^{\lambda t} P_t^\mathtt{G}1(x) dt
&=&\frac{1}{\lambda}\mathbb{E}_x[e^{\lambda \tau_\mathtt{G}\wedge N}-1]\\
\int_0^N e^{\lambda t} P_t^\mathtt{G}f(x) dt&\leq& \|f\|_\infty \frac{1}{\lambda}\mathbb{E}_x[e^{\lambda \tau_\mathtt{G}\wedge N}-1]
\end{eqnarray*}
The spectral calculus gives
\begin{eqnarray*}
\frac{1}{\lambda}\mathbb{E}_x[e^{\lambda \tau_\mathtt{G}\wedge N}-1]=\int_{[0,\infty[} \frac{e^{(\lambda-\xi) N}-1}{\lambda-\xi} dE_\xi 1
\end{eqnarray*}
and
\begin{eqnarray*}
\left(\frac{1}{\lambda}\mathbb{E}_x[e^{\lambda \tau_\mathtt{G}\wedge N}-1],\frac{1}{\lambda}\mathbb{E}_x[e^{\lambda \tau_\mathtt{G}\wedge N}-1]\right)
=\int_{[0,\infty[} \left(\frac{e^{(\lambda-\xi) N}-1}{\lambda-\xi}\right)^2 d(E_\xi 1,1)
\end{eqnarray*}

\begin{hyp}
$\lambda_0>0$ and for any $\lambda<\lambda_0$, $\mathbb{E}_x[e^{\lambda \tau_\mathtt{G}}]$ is an element of $\mathbb{L}^1(\mathbb{I}_\mathtt{G}\cdot m(dx))$.
\end{hyp}

\begin{theo}\label{Khasminskii-Friedman identity} Hypothesis($\lambda_0$) is equivalent to $E_{(\lambda_0-)}=0$, i.e. $-A^\mathtt{G}$ has a spectral gap of width at least equal to $\lambda_0$.
\end{theo}

\begin{rem}
This equality for bounded domains $\G$ is well-known since the works of Khasminskii~\cite{Khas} and Friedman~\cite{Fri}. However, the proof of~\cite[Theorem 2]{Khas} makes use of the boundedness of $\E_x\tau_\G$ in $\G$, which may not be the case in our general setting.
\end{rem}

The proof is divided in two parts.

\begin{lem} Hypothesis($\lambda_0$) implies $\int_{[0,\lambda_0[}dE_\xi=0$, i.e. $E_{(\lambda_0-)}=0$
\end{lem}

\begin{proof} Let $0<\lambda<\lambda_0$. For any bounded non negative function $f(x)\in\mathbb{L}^2(\mathbb{I}_\mathtt{G}\cdot m(dx))$, for all $\lambda> 0$, $0<N<\infty$, we can write
\begin{eqnarray*}
&&\frac{\|f\|_\infty^2}{\lambda}(\mathbb{E}_\bullet[e^{\lambda \tau_\mathtt{G}\wedge N}-1], 1)\\
&\geq &\left(\int_0^N e^{\lambda t} P_t^\mathtt{G}f dt, f\right)\\
&=&\int_{[0,\infty[}d(E_\xi f, f)\int_0^N e^{(\lambda-\xi) t} dt\\
&\geq&\int_{[0,\lambda[}d(E_\xi f, f)\int_0^N e^{(\lambda-\xi) t} dt\\
&=&\int_{[0,\lambda[}\frac{e^{(\lambda-\xi) N}-1}{\lambda-\xi} d(E_\xi f, f)\\
\end{eqnarray*}
Taking the limit when $N\uparrow\infty$, the preceding computation gives
$$
(E_{(\lambda-)}f,f)=\int_{[0,\lambda[}d(E_\xi f, f)=0
$$
The bounded non negative functions being dense in $\mathbb{L}^2(\mathbb{I}_\mathtt{G}\cdot m(dx))$, we conclude that $E_{(\lambda-)}=0$. Since this holds for any $0<\lambda<\lambda_0$, the lemma is proved.
\end{proof}

\begin{lem} Let $0<\lambda<\lambda_0$. Suppose that $E_{(\lambda_0-)}=0$, i.e. $\int_{[0,\lambda_0[}dE_\xi=0$. Then, $\mathbb{E}_\cdot[e^{\lambda}]$ is an element of $\mathbb{L}^2(\mathbb{I}_\mathtt{G}\cdot m(dx))$ and therefore Hypothesis($\lambda_0$) is true.
\end{lem}

\begin{proof} For $0<\lambda<\lambda_0$ we look at the formula
$$
\dcb
(\frac{1}{\lambda}\mathbb{E}_\bullet[e^{\lambda \tau_\mathtt{G}\wedge N}-1],\frac{1}{\lambda}\mathbb{E}_\bullet[e^{\lambda \tau_\mathtt{G}\wedge N}-1])
=\int_{[\lambda_0,\infty[} \left(\frac{1-e^{(\lambda-\xi) N}}{\xi-\lambda}\right)^2 d(E_\xi 1,1)
\dce
$$
Let $N\uparrow\infty$. The dominated convergence theorem yields
\begin{eqnarray*}
&&\frac{1}{\lambda^2}(\mathbb{E}_\bullet[e^{\lambda \tau_\mathtt{G}}-1], \mathbb{E}_\bullet[e^{\lambda \tau_\mathtt{G}}-1])\\
&=&\int_{[\lambda_0,\infty[}\frac{1}{(\xi-\lambda)^2}d(E_\xi 1, 1)\\
&\leq &\frac{1}{(\lambda_0-\lambda)^2}(1,1)\\
&=&\frac{1}{(\lambda_0-\lambda)^2}m(\mathtt{G})<\infty
\end{eqnarray*}
i.e. $\mathbb{E}_\cdot[e^{\lambda}]$ is effectively in $\mathbb{L}^2(\mathbb{I}_\mathtt{G}\cdot m(dx))$, and \`a fortiori in $\mathbb{L}^1(\mathbb{I}_\mathtt{G}\cdot m(dx))$ because $m$ is a bounded measure. The lemma is proved.
\end{proof}

Now, for $\G=]a,\infty[$ and $\G=]-\infty,a[$, in virtue of the ``all-or-none'' proposition~\ref{allornothing} the Hypotheses($\lambda_a^\pm$) are verified, and we obtain the equalities
\begin{equation*}
\gamma_a^+=\lambda_a^+\quad\mbox{and}\quad \gamma_a^-=\lambda_a^-.
\end{equation*}

To resume our main results, let us state a concluding theorem.
\begin{theo}\label{gapandmoments} For any $a\in\R$, 
$$
\frac{1}{A^+_a} = \gamma^+_a = \lambda^+_a\quad\mbox{and}\quad
\frac{1}{A^-_a} = \gamma^-_a = \lambda^-_a
$$
with
$$
B^+_a\leq A^+_a\leq 4B^+_a \quad\mbox{and}\quad
B^-_a\leq A^-_a\leq 4B^-_a.
$$
Furthermore,
$$
\sup_a\left(A^+_a\wedge A^-_a\right) \leq c_P\leq \inf_a\left(A^+_a\vee A^-_a\right)
$$
or, equivalently,
$$
\sup_a\left({\lambda^+_a}\wedge{\lambda^-_a}\right) \leq \gamma\leq \inf_a\left({\lambda^+_a}\vee{\lambda^-_a}\right),
$$
where $\gamma=1/c_P$ is the spectral gap of $X$ on $\R$.
\end{theo}

\section*{Acknowledgements}
The authors wish to thank Gilles Harg\'e, Francis Hirsch and Florent Malrieu for providing us useful information and bibliographic references.

\def\refname{References}

\end{document}